\documentclass[12pt,leqno]{amsart}
\usepackage{amssymb,amsthm,amsmath,latexsym}
\newtheorem{theorem}{\sc Theorem}[section]
\newtheorem{lemma}[theorem]{\sc Lemma}
\newtheorem{proposition}[theorem]{\sc Proposition}

\date{}

\title{On Profinite Groups with Engel-like Conditions }
\author{Raimundo Bastos }
\author{Pavel Shumyatsky }
\address{ Department of Mathematics, University of Brasilia,
Brasilia-DF, 70910-900 Brazil }
\email{bastos@mat.unb.br}
\email{pavel@unb.br}
\thanks{The first author was partially supported by CAPES and CNPq-Brazil; the second author
was supported by FAPDF and CNPq-Brazil}

\begin{document}
\begin{abstract} Let $G$ be a profinite group in which for every element $x\in G$ there exists a natural number $q=q(x)$ such that $x^q$ is Engel. We show that $G$ is locally virtually nilpotent. Further, let $p$ be a prime and $G$ a finitely generated profinite group in which for every $\gamma_k$-value $x\in G$ there exists a natural $p$-power $q=q(x)$ such that $x^q$ is Engel. We show that $\gamma_k(G)$ is locally virtually nilpotent. \\

\end{abstract}

\maketitle

\section{Introduction}
The positive solution of the Restricted Burnside Problem \cite{ze1,ze2} had led to many remarkable results on profinite groups. In particular, using Wilson's reduction theorem \cite{wil}, Zelmanov has been able to prove local finiteness of profinite periodic groups \cite{ze}. Recall that a group is periodic if all of its elements have finite order. The group $G$ is said to have a certain property locally if any finitely generated subgroup of $G$ possesses that property. Another result that was deduced following the solution of the Restricted Burnside Problem is that any profinite Engel group is locally nilpotent \cite{wize}. For elements $x,y$ of $G$ we define $[x,{}_1y]=[x,y]$ and $[x,{}_{i+1}y]=[[x,{}_{i}y],y]$ for $i\geq1$. The group $G$ is called an Engel group if for all $x,y\in G$ there is an integer $n=n(x,y)$ such that $[x,{}_{n}y]=1$. An element $y\in G$ is called Engel if for any $x\in G$ there is an integer $n=n(x)$ such that $[x,{}_{n}y]=1$; and $y$ is called $n$-Engel if for any $x\in G$ we have $[x,{}_{n}y]=1$. 

In the present article we consider profinite groups in which a power of any element is Engel. Our result can be viewed as a common generalization of both of the above results.

\begin{theorem}\label{a} Let $G$ be a profinite group in which for every element $x\in G$ there exists a natural number $q=q(x)$ such that $x^q$ is Engel. Then $G$ is locally virtually nilpotent.
\end{theorem} 

We recall that a profinite group possesses a certain property virtually if it has an open subgroup with that property.

Given an integer $k\geq1$, the word $\gamma_{k}=\gamma_k(x_1,\dots,x_k)$ is defined inductively by the formulae
\[
\gamma_1=x_1,
\qquad \text{and} \qquad
\gamma_k=[\gamma_{k-1},x_k]=[x_1,\ldots,x_k]
\quad
\text{for $k\ge 2$.}
\]
The subgroup of a group $G$ generated by all values of the word $\gamma_k$ is denoted by $\gamma_k(G)$. Of course, this is the familiar $k$-th term of the lower central series of $G$. When $G$ is a profinite group, $\gamma_k(G)$ denotes the closed subgroup generated by all values of the word $\gamma_k$. It was shown in \cite{shu} that if $G$ is a finitely generated profinite group in which all $\gamma_k$-values are Engel, then $\gamma_k(G)$ is locally nilpotent. In the present paper we establish the following related result.

\begin{theorem}\label{b} Let $p$ be a prime and $G$ a finitely generated profinite group in which for every $\gamma_k$-value $x\in G$ there exists a natural $p$-power $q=q(x)$ such that $x^q$ is Engel. Then $\gamma_k(G)$ is locally virtually nilpotent.
\end{theorem} 

We do not know whether the hypothesis that $G$ is finitely generated is really necessary in Theorem \ref{b}. The proof that we present here uses this assumption in a very essential way. Another natural question arising in the context of Theorem \ref{b} is whether the theorem remains valid with $q$ allowed to be an arbitrary natural number rather than $p$-power. This is related to the conjecture that if $G$ is a finitely generated profinite group in which every $\gamma_k$-value has finite order, then $\gamma_k(G)$ is locally finite (cf. \cite{shu}). Using the results obtained in \cite{shu} one can easily show that if $G$ is a finitely generated profinite group in which every $\gamma_k$-value has finite $p$-power order, then indeed $\gamma_k(G)$ is locally finite. This explains why Theorem \ref{b} is proved only in the case where $q$ is a $p$-power.

\section{Preliminary results}
As usual, if $\pi$ is a set of primes, we denote by $\pi'$ the set of all primes that do not belong to $\pi$. For a profinite group $G$ we denote by $\pi(G)$ the set of prime divisors of the orders of elements of $G$ (understood as supernatural, or Steinitz, numbers). If a profinite group $G$ has $\pi (G)=\pi$, then we say that $G$ is a pro-$\pi$ group. The maximal normal pro-$\pi$ subgroup of $G$ is denoted by $O_\pi(G)$. Recall that Sylow theorems hold for $p$-Sylow subgroups of a profinite group (see, for example, \cite[Ch.~2]{wil-book}). When dealing with profinite groups we consider only continuous homomorphisms and quotients by closed normal subgroups. If $X$ is a subset of a profinite group $G$, the symbol $\langle X\rangle$ stands for the closed subgroup generated by $X$. The symbol $\langle X^G\rangle$ stands for the normal closed subgroup generated by $X$. Throughout the paper a profinite pronilpotent group is just called pronilpotent. Of course, a pronilpotent group is a Cartesian product of its Sylow subgroups. 

\begin{lemma}\label{clo} Let $p$ be a prime and $G$ a pro-$p'$ group admitting an automorphism $a$ of finite $p$-power order. Then $\langle[x,a]^G\rangle=\langle[x,{}_ka]^G\rangle$ for any positive integer $k$ and any $x\in G$.
\end{lemma}
\begin{proof} Fix $k$ and $x\in G$. Set $M=\langle[x,a]^G\rangle$ and $N=\langle[x,{}_ka]^G\rangle$. It is obvious that $M$ contains $N$ and so we only need to show that $M\leq N$. We can pass to the quotient $G/N$ and without loss of generality assume that $[x,{}_ka]=1$. Thus, we need to show that $[x,a]=1$. Assume that $[x,a]\neq1$ and let $j$ be the minimal integer such that $[x,{}_ja]=1$. Set $y=[x,{}_{j-2}a]$ with the assumption that $y=x$ if $j=2$. Thus, we have $[y,a]\neq1$ but $[y,a,a]=1$. Since $[y,a]$ is a nontrivial $p'$-element commuting with $a$, the order of $[y,a]a^{-1}$ cannot be a $p$-power. On the other hand, it is clear that $[y,a]a^{-1}=a^{-y}$ is a conjugate of $a^{-1}$ and so the order of $[y,a]a^{-1}$ must be a $p$-power. This is a contradiction.
\end{proof}

\begin{lemma}\label{coen} Let $G$ be a pronilpotent group and assume that  $b$ and $c$ are elements of coprime orders in $G$. Set $a=bc$. Then $a$ is an $n$-Engel element if and only if so are both $b$ and $c$. The element $a$ is Engel if and only if both $b$ and $c$ are Engel.
\end{lemma}
\begin{proof} Assume that $a$ is $n$-Engel. Choose a Sylow $p$-subgroup $P$ of $G$. It is clear that if $p$ does not divide the order of $b$, then $[P,b]=1$. Assume that $p$ is a divisor of the order of $b$. In this case $p$ is coprime with the order of $c$ and so $[P,c]=1$. Therefore $[P,b]=[P,a]$ and $[P,{}_nb]=[P,{}_na]=1$. This happens for every choice of a Sylow subgroup $P$ of $G$. Since $G$ is pronilpotent, it follows that $[G,{}_nb]=1$. By a symmetric argument, $[G,{}_nc]=1$. Thus, if $a$ is $n$-Engel, then so are both $b$ and $c$.

Now assume that both $b$ and $c$ are $n$-Engel. Let $\pi_1$ be the set of primes dividing the order of $b$ and $\pi_2$ the set of primes dividing the order of $c$. Given $x\in G$ write $x=x_1x_2x_3$, where $x_1$ is a $\pi_1$-element, $x_2$ is a $\pi_2$-element and $x_3$ whose order is not divisible by primes in $\pi_1\cup\pi_2$. Since $G$ is pronilpotent, the elements  $x_1,x_2,x_3$ are uniquely determined by $x$. We have $$[x,{}_na]=[x_1,{}_nb]=[x_2,{}_nc]=1.$$ Thus, $a$ is $n$-Engel.

Further, using only obvious modifications of the above argument one can show that $a$ is Engel if and only if both $b$ and $c$ are Engel.
\end{proof}
\begin{lemma} \label{BBT}
Let $p$ be a prime and $G$ a pro-$p$ group generated by a set $X$. Suppose that $G$ can be generated by $m$ elements. Then there exist elements $x_1,\ldots,x_m \in X$ such that $G =\langle x_1,\ldots,x_m \rangle$.
\end{lemma}
\begin{proof}
A simple inverse limit argument shows that without loss of generality we can assume that $G$ is a finite $p$-group. Now existence of the elements $x_1,\ldots,x_m \in X$ is immediate from Bursnide Basis Theorem \cite[5.3.2]{Rob}.
\end{proof}
\begin{lemma}\label{mn} Let $m,n$ be positive integers and $G$ an $m$-generated pronilpotent group. Suppose that $G$ can be generated by $n$-Engel elements. Then $G$ can be generated by $m$ elements each of which is $n$-Engel. 
\end{lemma}
\begin{proof} Since $G$ is pronilpotent, $G$ is the direct product $P_1\times P_2\times\dots$, where the $P_i$ are Sylow subgroups of $G$. Let $X$ be a set of $n$-Engel elements such that $G$ is generated by $X$. For every element $x\in X$ write $x=x_1x_2\dots$, where $x_i$ is the projection of $x$ on $P_i$. Set $X_i=\{x_i\ \vert\ x\in X\}$. Thus, $X_i$ is the set of all $x_i$ where $x$ ranges over the set $X$. By Lemma \ref{coen} each set $X_i$ consists of $n$-Engel elements. It is clear that $P_i$ is generated by the set $X_i$. Since each Sylow subgroup $P_i$ can be generated by at most $m$ elements, Lemma \ref{BBT} implies that $P_i$ can be generated by at most $m$ elements from $X_i$. In each set $X_i$ choose $m$ (not necessarily distinct) elements $x_{i1},\dots,x_{im}$ such that $P_i=\langle x_{i1},\dots,x_{im}\rangle$. For every $j=1,\ldots,m$ we let $y_j$ be the product of $x_{ij}$, where $i$ ranges through the positive integers. By Lemma \ref{coen}, the elements $y_j$ are $n$-Engel. It is clear that $G=\langle y_1,\dots,y_m\rangle$ and so the lemma follows.

\end{proof}
\begin{lemma}\label{pi-n}
Let $a$ be an Engel element of a pronilpotent group $G$. Then there exist a positive integer $n$ and a finite set of primes $\pi$ such that $$ [O_{\pi'}(G), {}_n a] = 1.$$
\end{lemma}
\begin{proof} For each positive integer $i$ we set $$X_i=\{b\in G;\ [b,{}_i a] = 1 \}.$$  The sets $X_i$ are closed in $G$ and cover $G$. Therefore, by Baire's category theorem \cite[p. 200]{ke}, at least one of these sets contains non-empty interior. Hence, there exist a positive integer $n$, an open subgroup $H$ and $b\in G$ such that $[bh,{}_na]=1$ for every $h\in H$. Let $m$ be the index of $H$ in $G$, and let $\pi$ be the set of primes dividing $m$. Let $J$ be the product of the Sylow subgroups of $G$ corresponding to primes in $\pi$ and $K=O_{\pi'}(H)=O_{\pi'}(G)$. So $G=J\times K$ and hence without loss of generality we can assume that $b\in J$. Note that, $1=[bx,{}_na]=[b, {}_na][x,{}_na]$ for every $x\in K$. Therefore $[x,{}_na]=1$ for any $x\in K$. The result follows.   
\end{proof}

\begin{lemma}\label{pdash} Let $p$ be a prime and $k$ a positive integer. Assume that $G$ is a profinite group such that $G=O_{p'}(G)P$, where $P$ is a $p$-Sylow subgroup. Suppose that $O_{p'}(G)=\langle g_1,\dots,g_s\rangle$ and $P=\langle h_1,\dots,h_t\rangle$. Then $[O_{p'}(G),P]$ is generated by conjugates of commutators $[g_i,{}_kh_j]$ for $1\leq i\leq s$ and $1\leq j\leq t$.
\end{lemma}
\begin{proof} Let $D$ be the normal closure of all commutators $[g_i,{}_kh_j]$, where $1\leq i\leq s$ and $1\leq j\leq t$. By Lemma \ref{clo} this is the same as the normal closure of all commutators $[g_i,h_j]$. It is clear that $D\leq [O_{p'}(G),P]$ and so we can pass to the quotient $G/D$ and assume that $D=1$. In this case we have $G=\langle g_1,\dots,g_s\rangle\times\langle h_1,\dots,h_t\rangle$ and so $[O_{p'}(G),P]\leq D$. The lemma follows. 
\end{proof}

\begin{lemma}\label{ratio} Let $G$ be a profinite group and $k,n$ positive integers with $(n,p)=1$ for each $p\in\pi(G)$. Let $X$ be the set of all $\gamma_k$-values in $G$ and $Y$ the set of $n$-th powers of $\gamma_k$-values. Then $X=Y$.
\end{lemma} 
\begin{proof} A simple inverse limit argument shows that without loss of generality we can assume that $G$ is finite. Since $(|G|,n)=1$, it follows that $|Y|=|X|$. By \cite{gushu} $Y \subseteq X$. Hence $X=Y$.
\end{proof}
\begin{lemma}\label{finite} Let $G$ be a profinite group in which every ${\gamma_k}$-value has finite $p$-power order. If $G$ is generated by $p$-elements, then $G$ is a pro-$p$ group. If $G$ is generated by finitely many $p$-elements, then $G$ is a finite $p$-group.
\end{lemma} 
\begin{proof} Let $Q$ be any finite quotient of $G$. By Lemma 3.3 of \cite{shu2} we have ${\gamma_k}(Q)\leq O_p(Q)$. Since $Q$ is generated by $p$-elements, it follows that $Q$ is a finite $p$-group and hence $G$ is a pro-$p$ group. In particular, Theorem 1.2 of \cite{shu} tells us that $\gamma_k(G)$ is locally finite. Suppose now that $G$ is generated by finitely many $p$-elements. Then any nilpotent quotient of $G$ is finite and therefore $\gamma_k(G)$ is open. It follows that $\gamma_k(G)$ is finitely generated. We already know that $\gamma_k(G)$ is locally finite and so now we conclude that $\gamma_k(G)$ is finite. The lemma follows.
\end{proof}

The following result is a straightforward corollary of Theorem 1.4 in \cite{NS}.

\begin{theorem} \label{NS}
Let $k$ be a positive integer and $G$ a finitely generated profinite group. Then every element in $\gamma_k(G)$ is a product of finitely many $\gamma_k$-values. 
\end{theorem}

\section{Associated Lie algebras}
In the present section we will describe the Lie theoretical tools that are employed in the proofs of our results. 

Let $L$ be a Lie algebra over a field ${\mathfrak k}$.
We use the left normed notation: thus if
$l_1,l_2,\dots,l_n$ are elements of $L$, then
$$[l_1,l_2,\dots,l_n]=[\dots[[l_1,l_2],l_3],\dots,l_n].$$
An element $a\in L$ is called ad-nilpotent if there
exists a positive integer $n$ such that
$[x,{}_na]=0$ for all $x\in L$. If $n$ is the least 
integer with the above property then we say that $a$
is ad-nilpotent of index $n$. Let $X\subseteq L$ be
any subset of $L$. By a commutator in elements of $X$
we mean any element of $L$ that could be obtained from
elements of $X$ by means of repeated operation of
commutation with an arbitrary system of brackets
including the elements of $X$. Denote by $F$ the free
Lie algebra over ${\mathfrak k}$ on countably many free
generators $x_1,x_2,\dots$. Let $f=f(x_1,x_2,
\dots,x_n)$ be a non-zero element of $F$. The algebra
$L$ is said to satisfy the identity $f=0$ if
$f(l_1,l_2,\dots,l_n)=0$ for any $l_1,l_2,\dots,l_n
\in L$. In this case we say that $L$ is PI. We are now
in a position to quote a theorem of Zelmanov 
\cite[III(0.4)]{zelm} which has numerous important
applications to group theory.
\begin {theorem}\label{1} 
Let $L$ be a Lie algebra generated by $a_1,a_2,\dots,a_m$. Assume that $L$ is PI and that each commutator in the generators is ad-nilpotent. Then $L$ is nilpotent.
\end{theorem}

Let $G$ be a group and $p$ a prime. We denote by $D_i=D_i(G)$ 
the $i$-th dimension subgroup of $G$ in characteristic
$p$. These subgroups form a central series of $G$
known as the Zassenhaus-Jennings-Lazard series. Set
$L(G)=\bigoplus D_i/D_{i+1}$. 
Then $L(G)$ can naturally be viewed as a Lie algebra 
over the field ${\mathbb F}_p$ with $p$ elements.
The subalgebra of $L$ generated by $D_1/D_2$ will be 
denoted by $L_p(G)$. 
The following result is due to Lazard \cite{l2}.
\begin{theorem}\label{3} 
Let $G$ be a finitely generated pro-$p$
group. If $L_p(G)$ is nilpotent, then $G$ is
$p$-adic analytic.\end{theorem}
Let $x\in G$, and let $i=i(x)$ be the largest integer such
that $x\in D_i$. We denote by ${\tilde x}$ the 
element $xD_{i+1}\in L(G)$. We now cite two results providing
sufficient conditions for ${\tilde x}$ to be ad-nilpotent.
The following lemma is immediate
from the proof of Lemma in \cite[Section 3]{wize}.
\begin{lemma}\label{4} Let $x$ be an Engel element of a
profinite group $G$. Then ${\tilde x}$ is ad-nilpotent.
\end{lemma}
\begin{lemma} \label{5}(Lazard, \cite[page 131]{la}) 
For any $x\in G$ we have
$(ad\,{\tilde x})^p=ad\,(\widetilde {x^p})$. In particular,
if $x^q=1$, then ${\tilde x}$ is ad-nilpotent of index
at most $q$.
\end{lemma}
We note that $q$ in Lemma \ref{5} does not need to be a 
$p$-power. In fact it is easy to see that if $p^s$ is the 
maximal $p$-power dividing $q$, then $\tilde x$ is 
ad-nilpotent of index at most $p^s$. Combining Lemma \ref{4} 
and Lemma \ref{5}, one obtains
\begin{lemma}\label{45} Let $x$ be an element of a profinite 
group $G$ for which there exists a natural number $q$ such 
that $x^q$ is Engel. Then ${\tilde x}$ is ad-nilpotent.
\end{lemma}

Let $H$ be a subgroup of $G$ and $a_1,\dots,a_n\in G$. Let 
$w=w(x_1,\dots,x_n)$ be a nontrivial element of the 
free group with free generators $x_1,\dots,x_n$. Following 
\cite{wize} we say that the law $w=1$ is satisfied on the 
cosets $a_1H,\dots,a_nH$ if $w(a_1h_1,\dots,a_nh_n)=1$ 
for any $h_1,\dots,h_n\in H$. In \cite{wize} Wilson
and Zelmanov proved the following theorem.
\begin{theorem}\label{6} If $G$ is a group which has a
subgroup $H$ of finite index and elements $a_1,\dots,a_n$
such that a law $w=1$ is satisfied on the cosets
$a_1H,\dots,a_nH$, then for each prime $p$ the Lie algebra
$L_p(G)$ is PI.
\end{theorem}

\section{Proofs of the main results}

Throughout the rest of the paper $k$ stands for an arbitrary but fixed positive integer and $\bar{G} = G \times \ldots \times G$ ($k+1$ factors). First we will establish that if under the hypothesis of Theorem \ref{b} all numbers $q$ are coprime with the primes in $\pi(G)$, then $\gamma_k(G)$ is locally nilpotent. We start with the case where $G$ is a pro-$p$ group.
\begin{proposition}\label{prop} Let $G$ be a finitely generated pro-$p$ group such that for any ${\gamma_k}$-value $x\in G$ there exists a natural $p'$-number $q$ such that $x^q$ is Engel. Then $\gamma_k(G)$ is locally nilpotent.
\end{proposition}
\begin{proof} Since any element of $\gamma_k(G)$ can be expressed as a product of finitely many $\gamma_k$-values (Theorem \ref{NS}), it is sufficient to show that any subgroup $K$ generated by finitely many $\gamma_k$-values $a_1,\dots,a_s$ is nilpotent. 

Let $q_1,\dots,q_s$ be $p'$-numbers such that the elements $a_1^{q_1},\dots,a_s^{q_s}$ are Engel in $K$. It is clear that $K=\langle a_1^{q_1},\dots,a_s^{q_s}\rangle$. For each pair of positive integers $i,j$, where $j$ is a $p'$-number, we set $$S_{i,j}=\{(g_1,g_2,\dots,g_{k+1})\in\bar K; \ [g_1,{}_i[g_2,\dots,g_{k+1}]^j]=1\}.$$ Since the sets $S_{i,j}$ are closed in $\bar K$ and cover $\bar K$, by Baire's category theorem at least one of them contains non-empty interior. Therefore there exist an open subgroup $H$ of $K$, elements $b,b_1,\dots,b_k\in K$ and integers $n,q$ such that the cosets $bH,b_1H,\dots,b_kH$ satisfy the law $[y,{}_n[x_1,\dots,x_k]^q]\equiv1$.

By Theorem \ref{6} $L=L_p(K)$ is PI. Let $\tilde a_1,\dots, \tilde a_s$
be the homogeneous elements of $L$ corresponding to $a_1,\dots,a_s$. Since for any group commutator $h$ in $a_1,\dots,a_s$ there exists a $p'$-number $q$ such that $h^q$ is Engel, Lemma \ref{45} shows that any Lie commutator in $\tilde a_1,\dots,\tilde a_s$ is ad-nilpotent. Zelmanov's Theorem \ref{1} now tells us that $L$ is nilpotent. Therefore $K$ is $p$-adic analytic (Theorem \ref{3}). Obviously $K$ cannot contain a subgroup isomorphic to the free discrete group of rank two, so by the Tits' Alternative \cite{tits} $K$ has a soluble subgroup of finite index. Choose a dense subgroup $D$ of $K$ which is generated (as an abstract group) by finitely many Engel elements. By Plotkin's theorem \cite[Theorem 7.34]{rob} $D$ is nilpotent. It follows that $K$ is nilpotent, as well. The proof is complete. 
\end{proof}
\begin{theorem}\label{old} Let $\pi$ be a set of primes and $G$ a finitely generated pro-$\pi'$ group such that for any ${\gamma_k}$-value $x\in G$ there exists a natural $\pi$-number $q$ such that $x^q$ is Engel. Then $\gamma_k(G)$ is locally nilpotent.
\end{theorem}
\begin{proof} It will be convenient first to prove the theorem under the additional hypothesis that $G$ is pronilpotent. 

For each pair of positive integers $i,j$, where $j$ is a $\pi$-number, we set $$S_{i,j}=\{(g_1,g_2,\dots,g_{k+1})\in\bar G; \ [g_1,{}_i[g_2,\dots,g_{k+1}]^j]=1\}.$$ Arguing as in the proof of Proposition \ref{prop} we deduce that there exist an open subgroup $H$ of $G$, elements $b,b_1,\dots,b_k\in G$ and integers $n,q$ such that the cosets $bH,b_1H,\dots,b_kH$ satisfy the law $[y,{}_n[x_1,\dots,x_k]^q]\equiv1$. Let $m$ be the index of $H$ in $G$. We write $J$ for the product of the Sylow subgroups $P_1,\dots,P_r$ of $G$ corresponding to the primes dividing $m$ and $K$ for the product of the Sylow subgroups of $G$ corresponding to the primes not dividing $m$. Since $G=J\times K$ and since $G=JH$, without loss of generality we can assume that $b,b_1,\ldots,b_k\in J$ and it follows that $K$ satisfies the law $[y,{}_n[x_1,\dots,x_k]^q]\equiv1$. Lemma \ref{ratio} shows that $K$ actually satisfies the law $[y,{}_n[x_1,\dots,x_k]]\equiv1$. Hence $\gamma_k(K)$ is locally nilpotent by \cite[Theorem 1.1]{shu}. Further, Proposition \ref{prop} shows that $\gamma_k(P_i)$ is locally nilpotent for every $i=1,\dots,r$. We have $\gamma_k(G) = \gamma_k(P_1) \times \ldots \times \gamma_k(P_r) \times \gamma_k(K)$. Thus, it follows that in the case where $G$ is pronilpotent $\gamma_k(G)$ is locally nilpotent.

Now we drop the assumption that $G$ is pronilpotent. Since any element of $\gamma_k(G)$ can be expressed as a product of finitely many $\gamma_k$-values (Theorem \ref{NS}), it is sufficient to prove that any subgroup generated by finitely many $\gamma_k$-values $a_1,\dots,a_t$ is nilpotent. Set $H=\langle a_1,\dots,a_t\rangle$. We need to show that $H$ is nilpotent. Let $q_1,\dots,q_t$ be positive $\pi$-numbers such that $a_1^{q_1},\dots,a_t^{q_t}$ are Engel. It is clear that $H=\langle a_1^{q_1},\dots,a_t^{q_t}\rangle$. Since all Engel elements of a finite group lie in the Fitting subgroup \cite[12.3.7]{Rob}, it follows that $H$ is pronilpotent. Hence, by the previous paragraph, $\gamma_k(H)$ is locally nilpotent. The theorem now follows from the result of Plotkin just as in the proof of Proposition \ref{prop}.
\end{proof}
Now we will deal with the situation where the numbers $q$ are not necessarily coprime with the primes in $\pi(G)$.

\begin{proposition} \label{nq}
Let $k$ be positive integer and $G$ a finitely generated pronilpotent group in which for every $\gamma_k$-value $x$ there is an integer $q$ such that $x^q$ is Engel. Then $\gamma_k(G)$ is locally virtually nilpotent. 
\end{proposition} 
\begin{proof} For each pair of positive integers $i,j$ we set $$S_{i,j}=\{(g_1,\dots,g_{k+1})\in\bar G; \ [g_1,{}_i[g_2,\dots,g_{k+1}]^j]=1\}.$$ Arguing as in the proof of Proposition \ref{prop} we deduce that there exist an open subgroup $H$ of $G$, elements $b,b_1,\dots,b_k\in G$ and integers $n,q$ such that $$[bh, {}_n [b_1h_1, \ldots,b_kh_k]^q] = 1$$ for every $h,h_1,\ldots,h_k \in H$. Let $m$ be the index of $H$ in $G$. We write $J$ for the product of the Sylow subgroups of $G$ corresponding to the primes dividing $m$ and $K$ for the product of the Sylow subgroups of $G$ corresponding to the primes not dividing $m$. Since $G=J\times K$ and since $G=JH$, without loss of generality we can assume that $b,b_1,\ldots,b_k \in J$ and it follows that $K$ satisfies the law $[y,{}_n[x_1,\dots,x_k]^q]\equiv1$. Let $\pi$ be the set of prime divisors of $q$. Then $\gamma_k(O_{\pi'}(K))$ is locally nilpotent by Theorem \ref{old}. 

Now it suffices to show that $\gamma_k(P)$ is locally virtually nilpotent for every Sylow subgroup $P$ corresponding to a prime dividing $mq$. There is no loss of generality in assuming that $G$ is pro-$p$ group. Since every element of $\gamma_k(G)$ is a product of finitely many $\gamma_k$-values, it is sufficient to prove that any subgroup generated by finitely many $\gamma_k$-values is virtually nilpotent. Let $a_1,\dots,a_t$ be $\gamma_k$-values in $G$ and $M=\langle a_1,\dots,a_t \rangle$. Our argument will now imitate parts of the proof of Proposition \ref{prop}. By Theorem \ref{6} $L = L_p(M)$ is PI. Let $\tilde a_1,\dots, \tilde a_s$ be the homogeneous elements of $L$ corresponding to $a_1,\dots,a_s$. Since for any group commutator $h$ in $a_1,\dots,a_s$ there exists a positive integer $q$ such that $h^q$ is Engel, Lemma \ref{45} shows that any Lie commutator in $\tilde a_1,\dots,\tilde a_s$ is ad-nilpotent. Theorem \ref{1} now tells us that $L$ is nilpotent. Therefore, $M$ is $p$-adic analytic (Lemma \ref{3}). Obviously $M$ cannot contain a subgroup isomorphic to the free discrete group of rank two, so by the Tits' Alternative \cite{tits} $M$ has a soluble subgroup of finite index. Let $E$ be the closed subgroup of $M$ generated by all Engel elements in $G$ contained in $M$. Lemma \ref{finite} implies that $M/E$ is finite and therefore $E$ is finitely generated. Since $E$ is generated by all Engel elements contained in $M$ and since $E$ is finitely generated, Lemma \ref{BBT} implies that $E$ is generated by finitely many Engel elements. Choose an abstract dense subgroup $D$ in $E$ generated by finitely many Engel elements. By Plotkin's theorem \cite[Theorem 7.34]{rob}, $D$ is nilpotent.  Hence, also $E$ is nilpotent. The result follows.
\end{proof} 

The proof of Theorem \ref{a} is now easy.

\begin{proof}[Proof of Theorem \ref{a}] Recall that $G$ is a profinite group in which for every element $x\in G$ there exists a natural number $q=q(x)$ such that $x^q$ is Engel. We need to show that finitely generated subgroups of $G$ are virtually nilpotent.

Let $H$ be a finitely generated closed subgroup of $G$, and let $K$ be the subgroup generated by all Engel elements in $H$. By the hypothesis $H/K$ is a finitely generated periodic profinite group. According to Zelmanov's Theorem \cite{ze}, $H/K$ is finite. In particular, $K$ is a finitely generated pronilpotent group. By Proposition \ref{nq} applied with $k=1$, $K$ is virtually nilpotent. This completes the proof.      
\end{proof}

The proof of Theorem \ref{b} will be only somewhat more complicated.
\begin{proof}[Proof of Theorem \ref{b}] Recall that $G$ is a finitely generated profinite group in which some $p$-power of any ${\gamma_k}$-value is Engel. We wish to show that finitely generated subgroups of $\gamma_k(G)$ are virtually nilpotent.

The case where $k=1$ is covered by Theorem \ref{a}. So we will assume that $k \geq2$. Since any element of $\gamma_k(G)$ can be expressed as a product of finitely many $\gamma_k$-values (Theorem \ref{NS}), it is sufficient to prove that any subgroup generated by finitely many $\gamma_k$-values $a_1,\dots,a_t$ is virtually nilpotent. Let $H=\langle a_1,\dots,a_t\rangle$, and let $K$ be the closed subgroup of $H$ generated by all Engel elements of $H$. By Lemma \ref{finite} $H/K$ is a finite $p$-group and so $K$ is finitely generated. We will show that $K$ is nilpotent. It is clear that $K$ is pronilpotent and $O_{p'}(H)=O_{p'}(K)$. It follows that $O_{p'}(H)$ is finitely generated. Choose a Sylow $p$-subgroup $P$ in $H$. Since $H/O_{p'}(H)$ is isomorphic with $P$, we conclude that $P$ is finitely generated. Let $P=\langle h_1,\dots,h_r\rangle$ and $O_{p'}(H)=\langle g_1,\dots,g_s\rangle$. By Lemma \ref{pdash} $[O_{p'}(H),P]$ is generated by conjugates of $[g_i,{}_{k-1}h_j]$. Of course every element $[g_j,{}_{k-1}h_l]$ is a $\gamma_k$-value. Let $q_{jl}$ be positive $p$-powers such that $[g_j, {}_{k-1}h_l]^{q_{jl}}$ are Engel, for all $j \in \{1,\ldots,s\}$ and $l \in \{1,\ldots,r\}$. 

For each $i=1,\dots,t$ write $a_i=b_ic_i$, where $b_i$ is a generator of the $p'$-part of the procyclic subgroup $\langle a_i\rangle$ and $c_i$ is a generator of the $p$-part of the subgroup $\langle a_i\rangle$. Since $H/[O_{p'}(H),P]$ is pronilpotent, it follows that $O_{p'}(H)$ is generated by $[O_{p'}(H),P]$ and the elements ${b_1},\dots,{b_t}$. Let $q_1,\ldots,q_t$ be $p$-powers such that ${a_i}^{q_i}$ is Engel, for any $i=1,\ldots,t$. Thus, all elements ${b_1}^{q_1},\dots,{b_t}^{q_t}$ are Engel (Lemma \ref{coen}). Further, since all $q_i$ and $q_{jl}$ are $p$-powers, it follows that $O_{p'}(H)$ is generated by $[O_{p'}(H),P]$ and the elements ${b_1}^{q_1},\dots,{b_t}^{q_t}$. So $O_{p'}(H)$ is generated by ${b_1}^{q_1},\dots,{b_t}^{q_t}$ and conjugates of $[g_j,{}_{k-1}h_l]^{q_{jl}}$. Again, since we have only finitely many elements of this form, by Lemma \ref{pi-n} there exist a positive integer $n$ and a finite set of primes $\pi$ such that ${b_1}^{q_1}, \ldots,{b_t}^{q_t}$ and the elements $[g_j, {}_{k-1}h_l]^{q_{jl}}$ are $n$-Engel in their action on $O_{\pi'}(K)$. The subgroup $O_{\pi'}(K)$ is generated by $\pi'$-projections of the conjugates of ${b_1}^{q_1},\dots,{b_t}^{q_t}$ and $[g_j,{}_{k-1}h_l]^{q_{jl}}$. By Lemma \ref{coen} each of such projections is $n$-Engel in $O_{\pi'}(K)$. Since $O_{\pi'}(K)$ is finitely generated, it follows from Lemma \ref{mn} that $O_{\pi'}(K)$ is generated by finitely many $n$-Engel elements. By Theorem \ref{old}, ${\gamma_k}(O_{\pi'}(K))$ is locally nilpotent. Choose an abstract dense subgroup $D$ of $O_{\pi'}(K)$ generated by finitely many $n$-Engel elements. Since ${\gamma_k}(D)$ is locally nilpotent and $D$ is generated by finitely many $n$-Engel elements, we conclude that $D$ is nilpotent (Plotkin's theorem \cite[Theorem 7.34]{rob}). Hence $O_{\pi'}(K)$ is nilpotent as well. It therefore suffices to show that the $r$-Sylow subgroup $R$ of $K$ is nilpotent for any $r\in\pi$. Since $K$ is pronilpotent, $K=R\times O_{r'}(K)$ and so we can pass to the quotient $H/O_{r'}(K)$ and simply assume that $K=R$. Thus, $R$ is generated by finitely many Engel elements. If $r\neq p$, then by Proposition \ref{prop} $\gamma_k(R)$ is locally nilpotent and so by the theorem of Plotkin $R$ must be nilpotent.

Suppose $r=p$. Then $H$ is a pro-$p$ group. In this case we just repeat the final part of the proof of Proposition \ref{nq} and deduce that $H$ is virtually nilpotent. The theorem follows.  
\end{proof}

\end{document}